\documentclass[a4paper,12pt]{amsart}
\usepackage{}
\usepackage{mathrsfs}

\usepackage{amssymb}
\usepackage{latexsym}
\usepackage{amsfonts}
\usepackage{amsmath}
\usepackage{eucal}
\usepackage{bm}
\usepackage{bbm}
\usepackage{graphicx}
\usepackage[english]{varioref}
\usepackage[nice]{nicefrac}
\usepackage[all]{xy}
\usepackage{amsthm}

\newcommand{\supp}{\text {\rm supp}}

\newcommand{\Hom}{{\rm Hom}}

\newcommand{\res}{\text {\rm res}}
\newcommand{\pt}{\text {\rm pt}}

\def\i{^{-1}}

\def\<{\langle}
\def\>{\rangle}

\def\||{\parallel}
\def\d{\text{d}}

\def\a{\alpha}
\def\b{\beta}

\def\d{\delta}
\def\D{\Delta}

\def\th{\theta}

\def\l{\lambda}

\def\Om{\Omega}

\def\ZZ{\mathbb Z}
\def\NN{\mathbb N}

\def\CC{\mathbb C}

\def\cb{\mathcal B}

\def\cd{\mathcal D}

\def\cf{\mathcal F}

\def\cO{\mathcal O}
\def\cp{\mathcal P}

\def\tc{\tilde c}

\def\tW{\tilde W}

\def\tB{\tilde B}
\def\tG{\tilde G}
\def\tT{\tilde T}
\def\tX{\tilde X}

\def\fF{\mathfrak F}

\theoremstyle{plain}
\newtheorem{thm}{Theorem}[section]
\newtheorem*{thm*}{Theorem}
 \newtheorem{prop}[thm]{Proposition}
 \newtheorem{lem}[thm]{Lemma}
 \newtheorem{cor}[thm]{Corollary}

\theoremstyle{definition}

\theoremstyle{remark}

\newtheorem*{claim*}{Claim}

\begin{document}

\author{Sian Nie}
\address{}
\email{}
\title[]{The convolution algebra structure on $K^G(\cb \times \cb)$}

\begin{abstract}
We show that the convolution algebra $K^G(\cb \times \cb)$ is isomorphic to the Based ring of the lowest two-sided cell of the  extended affine Weyl group associated to $G$, where $G$ is a connected reductive algebraic group over the field $\CC$ of complex numbers and $\cb$ is the flag variety of $G$.
\end{abstract}

\maketitle

\section*{Introduction}

We are interested in understanding the equivariant group $K^G(\cb \times \cb)$, where $G$ is a connected reductive algebraic group over $\CC$ and $\cb$ is the flag variety of $G$.

When $G$ has simply connected derived subgroup, the K\"{u}nneth formula $K^G(\cb \times \cb) \simeq K^G(\cb) \otimes_{R_G} K^G(\cb)$ is proved in Proposition 1.6 of \cite {KL} and plays an important role in Kazhdan-Lusztig's proof of Delinge-Langlands conjecture for affine Hecke algebra associated to $G$, where $R_G$ denotes the representation ring of $G$. Furthermore, by Theorem 1.10 of \cite{Xi}, the convolution algebra structure on $K^G(\cb \times \cb)$ is isomorphic to the based ring of the lowest two-sided cell of the extended affine Weyl group associated to $G$.

In general, $K^G(\cb \times \cb)$ is not isomorphic to $K^G(\cb) \otimes_{R_G} K^G(\cb)$. To set a Deligne-Langlands-Lusztig classification for affine Hecke algebra associated to $G$, it seems useful to understand the equivariant $K$-groups $K^G(\cb \times \cb)$. The main result of this paper is Theorem \ref {main}, which says that the convolution algebra on $K^G(\cb \times \cb)$ is isomorphic to the based ring of the lowest two-sided cell of the extended affine Weyl group associated to $G$. Since the based ring is known explicitly in \cite{Xi}, the main result gives an explicit description to the equivariant $K$-group $K^G(\cb \times \cb)$.

\section{Preliminary}

\subsection{}
Let $G$ be a connected  reductive algebraic group over $\CC$, $B$ a Berol subgroup of $G$ and $T$ a maximal torus of $G$, such that $T \subset B$. The Weyl group $W_0=N_G(T)/T$ of $G$ acts on the character group $X=\Hom(T,\CC^\ast)$ of $T$. Using this action we define the extended affine Weyl group $W=X \rtimes W_0$.

By classification theorem for connected reductive algebraic groups, there exists a connected reductive algebraic group $\tG$ with simply connected derived subgroup such that $G$ is a quotient group of $\tG$ modulo a finite subgroup of the center of $\tG$. Denote by $\pi: \tG \rightarrow G$ the quotient homomorphism. Set $\tB=\pi \i (B)$, $\tT=\pi \i (T)$, $\tX=\Hom(\tT,\CC^\ast)$ and $\tW=\tX \rtimes W_0$. Note that $X$ is naturally a subgroup of $\tX$ of finite index, hence $W$ is a naturally subgroup of $\tW$ of finite index.

Let $R \subset X$ be the root of $G$ and $\tG$. Let $R^- \subset R$ to be the set of negative roots determined by $B$. Set $R^+ = R - R^-$. Let $\D \subset R^+$ be set of simple positive roots.

Denote by $\l_\a$ the dominant fundamental weight corresponding to a simple positive root $\a \in R^+$. For any $w \in W_0$, define $x_w=w \i (\prod_{\a \in \D, w \i (\a)<0}\l_\a) \in \tX$. It is known that $\ZZ[\tX]$ is a free $\ZZ[X]^{W_0}$-module with a basis $\{x_w | w \in W_0\}$.

Let $\ell: \tW \rightarrow \NN$ be the length function. Note that $\ell(w\l)=\ell(w)+\ell(\l)$ for any $w \in W_0$ and any dominant weight $\l \in \tX$. Also we have $\ell(\l_\a s_\a)=\ell(\l_\a)-1$ for any positive simple root $\a \in \D$.

\subsection{}
Let $\Sigma=\{wx_w | w \in W_0\}$. Then the lowest two-sided cell $\tc_0$ of $\tW$ consists of elements $f \i w_0 \chi g$ with $f,g \in \Sigma$ and $\chi \in \tX^+$. (See \cite {Shi}) Here $w_0$ is the longest element of $W_0$ and $\tX^+$ is the set of dominant weights of $\tX$.  The lowest two-sided cell of $W$ is $c_0=\tc_0 \cap W$. The ring structure of $J_{\tc_0}$ of $\tc_0$ is defined in \S2 of \cite{L1} and explicitly determined in \S4 \cite{Xi}. As a $\ZZ$-module, it is free with a basis $t_z, z \in \tc_0$. The based ring $J_{c_0}$ of $c_0$ is a subring of $J_{\tc_0}$ spanned by all $t_z, z \in c_0$.

\subsection{}
For an algebraic group $M$ over $\CC$ and an variety $Z$ over $\CC$ which admits an algebraic action of $M$, denote by $K^M(Z)$ the Grothendieck group of $M$-equivariant coherent sheaves on $Z$. We refer to Chapter 5 of \cite{CG} for more about the equivariant $K$-group $K^M(Z)$.

There is a natural map $L: \ZZ[\tX] \rightarrow K^{\tG}(\cb)$ which associates $\chi \in \tX$ to the unique equivariant line bundle $[L(\chi)]$ on $\cb$ such that $\tT$ acts on the fibre $L(\chi)_B$ over $B \in \cb$ via $\chi$. Here $\cb=\tG / \tB=G /B$ is the flag variety. It is well known that $L$ is an isomorphism. By abuse of notation, we will use $\chi$ and $[L(\chi)]$ interchangeably in the following context.

The convolution on $K^G(\cb \times \cb)$ is defined by $$\fF \ast \fF'={Rp_{13}}_*(p^*_{12}\fF \bigotimes_{\cO_{\cb \times \cb \times \cb}}^L p^*_{23}\fF'),$$ where $\fF, \fF' \in K^G(\cb \times \cb)$ and $p_{12}, p_{13}, p_{23}: \cb \times \cb \times \cb \rightarrow \cb \times \cb$ are obvious natural projections. Identifying $K^{\tG}(\cb \times \cb)$ with $K^{\tG}(\cb) \otimes_{R_{\tG}} K^{\tG}(\cb) \simeq \ZZ[\tX] \otimes_{\ZZ[\tX]^{W_0}} \ZZ[\tX]$, the convolution becomes $$(\chi_1 \otimes \chi_2) \ast (\chi_1' \otimes \chi_2')=(\chi_2,\chi_1')\chi_1 \otimes \chi_2',$$ where $(,): \ZZ[\tX] \otimes_{\ZZ[\tX]^{W_0}}\ZZ[\tX] \rightarrow \ZZ[\tX]^{W_0}$ is given by $$(\chi_2, \chi_1')=\d \i \sum_{w \in W_0}(-1)^{\ell(w)}w(\chi_2 \rho \chi_1').$$ Here $\d=\prod_{\a \in R^+}(\a^{\frac{1}{2}}-\a^{-\frac{1}{2}})$ and $\rho=\prod_{\a \in R^+}\a^{\frac{1}{2}}$.

For $f=wx_w \in \Sigma$, set $x_f=x_w$. Since $(,)$ is a perfect pairing (See Proposition 1.6 in \cite{KL}), we can find $y_f \in \ZZ[\tX]$ such that $(x_f, y_f)=\d_{f,f'}$. The following result is due to N. Xi. (See Theorem 1.10 in \cite{Xi}.)\\

$(*)$ {\it The map $\sigma: J_{\tc_0} \rightarrow K^{\tG}(\cb \times \cb) \simeq \ZZ[\tX] \otimes_{\ZZ[\tX]^{W_0}} \ZZ[\tX]$ given by $t_{f \i w_0 \chi f'} \mapsto V(\chi) y_f \otimes x_{f'}$ for $\chi \in \tX^+$ and $f,f' \in \Sigma$ is an isomorphism of $R_{\tG}$-algebras. Here $V(\chi) \in R_{\tG}$ stands for the irreducible $\tG$-module with highest weight $\chi$.}\\

Now we state the main result of this paper.
\begin{thm}\label{main}
(a) The natural map $i: K^G(\cb \times \cb) \rightarrow K^{\tG}(\cb \times \cb)$ is an injective of homomorphism of algebra.\\
(b) As a $\ZZ$-module, the image of $i$ is spanned by $\{V(\chi)y_f \otimes x_{f'}; \chi \in \tX^+, f,f' \in \Sigma, f \i w_0 \chi f' \in W\}$.\\
(c) In particular, via the isomorphism $\sigma$ in (*), $J_{c_0}$ is isomorphic to the convolution algebra $K^G(\cb \times \cb)$ as $R_G$-algebras.
\end{thm}

\section{Proof of Theorem \ref{main}}
\subsection{}\label{fact}
Set $\Om=\tW/W=\tX/X=\{\l X; \l \in \tX\}$, which is a finite abelian group. For a left coset $\l X$, let $\ZZ[\l X]$ be the $\ZZ$-submodule of the group algebra $\ZZ[\tX]$ spanned by elements in $\l X$. For any $A \in \ZZ[\l X]$ and $B \in \ZZ[\mu X]$, we have $AB \in \ZZ[\l \mu X]$. Moreover if there is $C \in \ZZ[\tX]$ such that $A=BC$, then $C \in \ZZ[\l \mu \i X]$.

\begin{lem}\label{lem}
For $f \in \Sigma$, we have $y_f \in \ZZ[x_f \i X]$.
\end{lem}
\begin{proof}
For $f,f' \in \Sigma$, set $A_{f,f'}=(x_f, x_{f'})=\d \i \sum_{w \in W_0}(-1)^{\ell(w)}w(x_f \rho x_{f'})$ which lies in $ \ZZ[x_f x_{f'} X]$. Let $(A^{f,f'})_{(f,f') \in \Sigma \times \Sigma}$ be the inverse matrix of $(A_{f,f'})_{(f,f') \in \Sigma \times \Sigma}$. Then a direct computation shows $A^{f,f'} \in \ZZ[x_f \i x_{f'} \i X]$. Since $y_f=\sum_{f' \in \Sigma}A^{f,f'} x_{f'}$, we have $y_f \in \ZZ[x_f \i X]$.
\end{proof}

\subsection{}\label{strata}
For $w \in W_0$, let $Y_w \subset \cb \times \cb$ be the $G$-orbit containing $(B,wB)$. Then $\cb \times \cb=\coprod_{w \in W_0}Y_w$ and the projection to the first factor $p_1: Y_w \rightarrow \cb$ is an affine bundle of rank $\ell(w)$. Numbering the elements of $W_0$ as $u_1, u_2, \cdots, u_r$ such that $u_i \nless u_j$ if $j<i$. Let $F_i=\coprod_{j \leq i}Y_{u_j}$. Then $F_i$ is closed in $\cb \times \cb$. We have the following commutative diagram

\begin{equation*}
\xymatrix{
0 \ar[r] & K^G(F_{i-1}) \ar[d] \ar[r] & K^G(F_i) \ar[d] \ar[r] & K^G(Y_{u_i}) \ar[d] \ar[r] & 0\\
0 \ar[r] & K^{\tG}(F_{i-1}) \ar[r] & K^{\tG}(F_i) \ar[r] & K^{\tG}(Y_{u_i}) \ar[r] & 0 }
\end{equation*}

where all the morphisms are natural and obvious. By \cite{CG}, the rows of the diagram are exact sequences. Since $p_1: Y_{u_i} \rightarrow \cb$ is an affine bundle, we have the following commutative diagram
\begin{equation*}
\xymatrix{
K^G(Y_{u_i}) \ar[r]^\sim & K^G(\cb) \ar[r] \ar[d]_\wr & K^{\tG}(\cb) \ar[r]^\sim \ar[d]_\wr & K^{\tG}(Y_{u_i})\\
                             & \ZZ[X] \ar[r] & \ZZ[\tX] }
\end{equation*}
which shows that the natural morphism $K^G(Y_{u_i}) \rightarrow K^{\tG}(Y_{u_i})$ is injective. Using induction on $i$, we see that the natural morphism $i: K^G(\cb \times \cb) \rightarrow K^{\tG}(\cb \times \cb)$ is injective. One shows directly that $i$ is a homomorphism of convolution algebras. Part (a) of Theorem \ref{main} is proved.

\subsection{}
For $w \in W_0$, let $X_w=\tB w \tB/ \tB$ and $X^w = w B^+ \tB / \tB$, where $\tB^+ \supset \tT$ is the opposite of $\tB$. Then $X^w$ is an $\tT$-invariant open neighborhood of $X_w$ in $\cb$. Set $X_i=\coprod_{j \leq i}X_{u_j}$. Note that $\cb \times \cb = \tB \backslash (\tG \times \cb)$, where the action of $\tB$ on $\tG \times \cb$ is given by $b (g, h \tB)=(g b \i, bh \tB)$. Hence $K^{\tG}(\cb \times \cb) \simeq K^{\tB}(\cb) \simeq K^{\tT}(\cb)$. Similarly, $K^{\tG}(F_i) \simeq K^{\tT}(X_i)$ and $K^{\tG}(Y_w) \simeq K^{\tT}(X_w)$.

Let $j_w: Y_w \rightarrow \cb \times \cb$ be the natural $G$-equivariant inclusion. Since $X_w$ is a $\tT$-equivariant vector bundle over a single point. We have $K^{\tG}(Y_w) \simeq K^{\tT}(X_w) \simeq \ZZ[\tX]$. Then a direct computation shows that the induced homomorphism $Lj_w^*$ of equivariant $K$-groups is given by
\begin{align*}
Lj_w^*: K^{\tG}(\cb \times \cb)& \rightarrow  K^{\tG}(Y_w) \simeq  \ZZ[\tX],\\
             x_1 \otimes x_2   & \mapsto x_1 w(x_2),
\end{align*}
where $x_1 \otimes x_2 \in \ZZ[\tX] \otimes_{\ZZ[\tX]^{W_0}} \ZZ[\tX] \simeq K^{\tG}(\cb \times \cb)$.

\begin{prop}\label{criterion}
Let $l \in K^{\tG}(\cb \times \cb)$, then $l \in K^G(\cb \times \cb)$ if and only if $Lj_w^*(l) \in K^G(Y_w)$ for any $w \in W_0$.
\end{prop}
\begin{proof}
Denote by $k_w: X_w \hookrightarrow \cb$, $k_i: X_i \hookrightarrow \cb$ and $j_i: F_i \hookrightarrow \cb \times \cb$ the natural immersions. Then we have
\begin{equation*}
\xymatrix{
K^{\tG}(\cb \times \cb) \ar[r]^{Lj_i^*} \ar[d]_\wr & K^{\tG}(F_i) \ar[r]^{\res} \ar[d]_\wr & K^{\tG}(Y_{u_i}) \ar[d]_\wr \\
K^{\tT}(\cb) \ar[r]^{Lk_i^*} & K^{\tT}(X_i) \ar[r]^{\res} & K^{\tT}(X_{u_i}) }
\end{equation*}

Hence our proposition is equivalent to the following statement\\

{\it For any $l \in K^{\tT}(\cb)$, $l \in K^T(\cb)$ if and only if $Lk_w^*(l) \in K^T(X_w)$ for any $w \in W_0$. }\\

The "only if" part is obviously. We show the "if" part. Note that the support $\supp(l)$ of $l$ belongs to some $X_i$, we argue by induction on $i$. If $\supp(l)=\emptyset$, that is, $l=0$, then the statement follows trivially. Suppose the proposition holds for any element in $K^{\tT}(\cb)$ whose support belongs to $X_{i-1}$. We show it also holds for $z \in K^{\tT}(\cb)$ with $\supp(l) \subset X_i$.

Since we have the following $\tT$-equivariant morphism
\begin{equation*}
\xymatrix{
X_{u_i} \ar[r] \ar[d]_\wr & X^{u_i} \ar[d]_\wr \\
\bigoplus_{\a \in R^-, u_i\i(\a)>0}\CC_{\a} \ar[r] & \bigoplus_{\b \in R^+} \CC_{u_i(\b)}, }
\end{equation*}
where $\CC_\a$ denotes the one dimensional vector space $\CC$ on which $\tT$ acts via the character $\a$. By Proposition 5.4.10 in \cite{CG},
$$Lk_{u_i}^*(l)=\prod_{{\a \in R^+, u_i\i(\a)>0}}(1- \a \i)l|_{X_{u_i}} \in K^T(X_{u_i}).$$
Since $\prod_{{\a \in R^+, u_i\i(\a)>0}}(1- \a \i) \in K^T(X_{u_i})$, then $l|_{X_{u_i}} \in K^T(X_{u_i})$ by \ref{fact}. Since $X_{u_i}$ is a $\tT$-stable open subset of $X_i$, there exists $l' \in K^T(X_i)$ such that $l'|_{X_{u_i}}=l|_{X_{u_i}}$. Then $\supp(l-l') \subset X_{i-1}$. Using induction hypothesis, we have $l-l' \in K^T(\cb)$. Hence $l=(l-l')+l' \in K^T(\cb)$ and the proof is finished.
\end{proof}

\begin{cor}\label{cor}
Let $z=f \i w_0 \chi f'$ with $f,f' \in \Sigma$ and $\chi \in \tX^+$. Then $z \in W$ if and only if $\sigma(z)=V(\chi) y_f \otimes x_{f'} \in K^G(\cb \times \cb)$.
\end{cor}
\begin{proof}
Obviously $z \in W$ if and only if $x_f \i \chi x_f \in X$. On the other hand, $V(\chi) \in \ZZ[\chi X]$. By Lemma \ref{lem}, $y_f \in \ZZ[x_f \i X]$. Then by Proposition \ref{criterion}, $V(\chi) y_f \otimes x_{f'} \in K^G(\cb \times \cb)$ if and only if $V(\chi) y_f x_{f'} \in \ZZ[X]$, which is equivalent to $x_f \i \chi x_f \in X$.
\end{proof}

\begin{proof}[Proof of part (b) and (c) of Theorem \ref{main}]
By Corollary \ref{cor}, $\sigma(J_0) \subset K^G(\cb \times \cb)$. It remains to show that $K^G(\cb \times \cb) \subset \sigma(J_0)$. Let $l \in K^G(\cb \times \cb)$. Due to $(*)$, we can assume that $$l=\sum_{f,f' \in \Sigma}a_{f,f'}V(\chi_{f,f'})y_f \otimes x_{f'}$$ with $a_{f,f'} \in \ZZ$, $\chi_{f,f'} \in \tX^+$. Since $y_f \otimes x_f = \sigma(t_{f \i w_0 f}) \in K^G(\cb \times \cb)$ for each $f \in \Sigma$, we have
$$(y_f \otimes x_f) \ast l \ast (y_{f'} \otimes x_{f'})=a_{f,f'} V(\chi_{f,f'}) y_f \otimes x_{f'} \in K^G(\cb \times \cb).$$ Hence by Corollary \ref{cor}, $f \i w_0 \chi_{f,f'} f' \in J_0$ whenever $a_{f,f'} \neq 0$. Hence $l=\sum_{f,f'}a_{f,f'}\sigma(t_{f \i w_0 \chi_{f,f'} f'}) \in \sigma(J_0)$.
\end{proof}

\section{Some results on $K^G(\mathcal{P}\times\mathcal{P})$}
\subsection{}
Let $I \subset \D $ be a subset and $P$ the parabolic subgroup of type $I$ containing $B$. Define $\cp=G/P$ be the variety of all parabolic subgroups of type $I$.

Let $\cd$ be the set of double cosets of $W_0$ with respect to $W_I\subset W_0$. Here $W_I$ is the parabolic subgroup generated by $I$. For each $w \in W_0$, define
$$Z_{\bar{w}}=\{(P,P') \in \cp \times \cp \ |\  (P,P')\text{ is conjugate to } (P,{}^w P)\},$$
where $\bar{w}$ denotes the double coset $W_I w W_I$. Then we have
$$\cp \times \cp=\coprod_{d \in \cd}Z_{d}.$$

For any double coset $d \in \cd$, there is a unique element $u_d \in d$ such that $u_d$ is the smallest in $d$ under the Bruhat order. Let $d, d'\in \cd$, we say $d \geq d'$ if $u_d \geq u_{d'}$ under the Bruhat order.

\begin{lem}
With notations as above, then we have $d \geq d'$ if and only if $\bar Z_d \supset Z_{d'}$.
\end{lem}
\begin{proof}
Consider the natural projections $Y_{u_d}\rightarrow Z_d$ and $Y_{u_{d'}}\rightarrow Z_{d'}$, which are restrictions of the natural projection $p:\cb \times \cb \rightarrow \cp \times \cp$ to $Y_{u_d}$ and $Y_{u_{d'}}$ respectively. Since $u_d \geq u_{d'}$, we have $\bar Y_{u_d} \supset Y_{u_{d'}}$. Hence $Z_{d'} \subset p(\bar{Y}_{u_{d'}})\subset p(\bar{Y}_{u_d})=\bar{Z_d}$. The ``if part'' follows from the fact that the morphism $p$ above is projective.
\end{proof}

\begin{prop}\label{exact}
Let $d \in \cd$. We have the following short exact sequence:
$$0 \rightarrow K^G(Z_{<d})\rightarrow K^G(\bar Z_d)\rightarrow K^G(Z_d)\rightarrow 0.$$ Here $Z_{< d}= \bar Z_d - Z_d$.
\end{prop}
\begin{proof}
It suffices to prove the injection $K^G(Z_{<d})\hookrightarrow K^G(\bar Z_d)$. Note that we have a natural injective morphism $K^G(Z)\hookrightarrow K^T(Z)$ for $Z=Z_{<d}$ or $Z=\bar Z_d$, where $T \subset G$ is a maximal torus. Hence it suffices to prove the result for torus $T$. By Lemma 1.6 in \cite{L2}, we just have to show that $K^T(Z_d)$ is a free $R_T$-module for all $d \in \cd$. Let's consider the projection $q: Z_d \rightarrow \cp$ given by $(P,P')\mapsto P$. Let $x \in \{w \in W_0 \ |\  w \text{ is of minimal length among } w W_I\}$. Define $Z_d^x=q \i (B x P/P)$ which is $T$-stable. Thus it suffices to show $K^T(Z_d^x)$ is a free $R_T$-module. Note that $Z_d^x=B_x \backslash (B \times F_x)$, where $B_x =\{b \in B; {}^{bx} P={}^x P\}$ and $F_x=q \i ({}^x P)$. Then we have
$$K^T(Z_d^x)=K^T(B_x \backslash (B \times F_x))=K^{B}(B_x \backslash (B \times F_x))=K^T(F_x).$$
Since $F_x={}^x P/({}^x P \cap {}^{xu_d} P)$, and ${}^x P /({}^x P \cap {}^{xu_d} P)$ contains a Borel subgroup of a Levy subgroup of ${}^x P$ containing $T$, it admits a partition which are $T$-vector spaces. Hence the proposition follows.
\end{proof}

\subsection{}\label{rmk}
Let $w \in d = W_I w W_I$. Choose Borel subgroups $\hat{B}_w \subset P$ and $\hat{B}_w'\subset {}^w P$ such that $\hat{B}_w \cap \hat{B}_w'$ is a Borel subgroup of $P \cap {}^w P$. Assume $(\hat{B}_w, \hat{B}_w') \in Y_u$ for some $u \in d$. Consider the natural projection $p|_{Y_u}:Y_u \rightarrow Z_{\bar{w}}$. It is easy to see $p^{-1}(P, {}^w P)=P \cap {}^w P / (\hat{B}_w \cap \hat{B}_w')=\cb_{P \cap {}^wP}$. Here $\cb_{P \cap {}^w P}$ denotes the flag variety of $P \cap {}^w P$. Hence $p|_{Y_{u}}$ is a projective morphism. So $u = u_d$ is the minimal lehgth element in the double coset $d$.

Define
\begin{align*}
Rp_*: K^G(Y_{u_d})&\longrightarrow K^G(Z_{\bar{w}})\\
[\cf]&\mapsto \sum_i(-1)^i[R^ip_* \cf].
\end{align*}

Let $\chi \in X=\Hom(T,\CC^*)$. Denote by $\th_\chi$ the $G$-equivariant line bundle over $Y_{u_d}$ such that $T$ acts on the fiber of $\th_\chi$ over $(\hat{B}_w,\hat{B}_w') \in Y_{u_d}$ via $\chi$.

\begin{prop}\label{sur}
With notations in \ref{rmk}. The morphism ${R(p|_{Y_{u_d}})}_*: K^G(Y_{u_d})\longrightarrow K^G(Z_{\bar{w}})$ defined above is surjective.
\end{prop}

\begin{proof}
Let's compute ${R(p|_{Y_{u_d}})}_*([\th_\chi])$. Note that $p$ is smooth and projective, and $Z_{\bar{w}}$ is integral (as a scheme). Hence by Corollary 12.9 of \cite{Har}, we have ${R^ip|_{Y_{u_d}}}_*(\th_\chi)$ is vector bundle over $Z_{\bar{w}}$ and $R^ip_*(\th_\chi)|_{(P,{}^wP)}=H^i(p \i (P, {}^w P),\th_\chi|_{p \i (P, {}^w P)})$ for all $i \geq 0$. Hence when $\chi$ is a dominant weight, $Rp_*(\th_\chi)|_{(P, {}^w P)}=V_\chi$, where $V_\chi$ is the irreducible $P \cap {}^w P$-module with highest weight $\chi$. Note that all $[V_\chi]$ with $\chi$ dominant generates $K^{P \cap {}^w P}(\pt)=K^G({Z_{\bar{w}}})$. Hence $Rp_*$ is surjective.
\end{proof}

\begin{cor}
The natural morphism $Rp_\ast: K^G(\mathcal{B}\times\mathcal{B})\rightarrow K^G(\mathcal{P}\times\mathcal{P})$ is surjective.
\end{cor}
\begin{proof}
Let $l\in K^G(\mathcal{P}\times\mathcal{P})$. We show that $l$ lies in the image of $Rp_\ast $ by induction on the dimension of the support of $\supp(l)$. If $\supp(l)=\emptyset$, hence $l=0$, it follows trivially. Now we assume that the statement holds for any $l'$ such that $\supp(l')\subset Z_{<d}$ for some $d \in \cd$ and any $l'\in K^G(\cp \times \cp)$. Assume $\supp(l)\subset Z_{\leq d}$, Let $l_d=i_d^*(l) \in K^G(Z_d)$, where $i_d: Z_d \hookrightarrow Z_{\leq d}$ be the natural open immersion. By Proposition \ref{sur}, there is $f_d \in K^G(Y_{u_d})$ such that ${R(p|_{Y_{u_d}})}_*(f_d)=l_d$. Now extend $f_d$ to some $\bar{f}_d \in K^G(\bar{Y}_{u_d})$. Thanks to 3.2, we have $p^{-1}(Z_d)\bigcap \bar{Y}_{u_d}=Y_{u_d}$. Hence $Rp_*(\bar{f}_d)|_{Z_d}=l_d$. Then the support $i_d^*(l-Rp_*(\bar{f}_d))=0$. By Proposition \ref{exact}, we have $\supp(l-Rp_*(\bar{f}_d))\subset Z_{<d}$. By induction hypothesis, the statement follows.
\end{proof}

\end{document}